\documentclass[12pt]{amsart}

\usepackage{bookman} % font 
\usepackage{paralist} % more variants for itemize

\usepackage{amssymb} % The amssymb package provides various useful mathematical symbols
\usepackage{amsmath,tikz,wasysym} % Additional Packages

\newtheorem{proposition}{Proposition}[section]
\newtheorem{theorem}{Theorem}[section]
\newtheorem{example}{Example}[section]

\newtheorem{remark}{Remark}[section]
\newtheorem{definition}{Definition}[section]
\newtheorem{corollary}{Corollary}[section]
\newtheorem{lemma}{Lemma}[section]
\newcommand{\npmatrix}[1]{\left( \begin{matrix} #1 \end{matrix} \right)}
\newcommand{\R}{\mathbb{R}}

\newcommand{\mc}{\mathcal}
\newcommand{\mr}{\mathrm{mr}}
\newcommand{\rk}{\mathrm{rk}}

\begin{document}
\title{Graphs that allow all the eigenvalue multiplicities to be even}
\author{Polona Oblak, Helena \v Smigoc}

\address{P.~Oblak: Faculty of Computer and Information Science, University of Ljubljana, Tr\v za\v ska 25, SI-1000 Ljubljana, Slovenia; e-mail: polona.oblak@fri.uni-lj.si}
\address{H.~\v Smigoc:~School of Mathematical Sciences, University College Dublin, Belfield, Dublin 4, Ireland; email: helena.smigoc@ucd.ie}

\bigskip

\begin{abstract}
 Let $G$ be an undirected graph on $n$ vertices and let $S(G)$ be the set of all $n \times n$ real symmetric matrices whose nonzero off-diagonal entries occur in exactly the positions corresponding to the edges of $G$. The inverse eigenvalue problem for a graph $G$ is a problem of determining all possible lists that can occur as the lists of eigenvalues of matrices in $S(G).$ This question is, in general, hard to answer and several variations were studied, most notably the minimum rank problem. In this paper we introduce the problem of determining for which graphs $G$ there exists a matrix in $S(G)$ whose characteristic polynomial is a square, i.e. the multiplicities of all its eigenvalues are even. We solve this question for several families of graphs. 
\end{abstract}

\maketitle 

\section{Introduction}

Given a simple undirected graph $G=(V(G),E(G))$ with vertex set $V(G)=\{1,2,\ldots,n\}$, let $S(G)$ be the set of all real symmetric $n \times n$ matrices $A=(a_{ij})$ such that, for $i \neq j$, $a_{ij} \neq 0$ if and only if $(i,j) \in E(G).$ There is no restriction on the diagonal entries of $A$.   

The question of characterizing all lists of real numbers $$\{\lambda_1,\lambda_2,\ldots,\lambda_n\}$$ that can be the spectrum of a matrix $A \in S(G)$ is known as the \emph{Inverse eigenvalue problem for $G$}, and it is hard to answer in general. This question and the related question of characterizing all possible multiplicities of eigenvalues of matrices in $S(G)$ has been studied primarily for trees \cite{MR2996934,MR1902112,MR3005270,MR3034496}.
The subproblem to the inverse eigenvalue problem for graphs that has attracted a lot of attention over the recent years is that of minimizing the rank of all $A \in S(G)$. Finding the minimal rank of $G$, defined as
 $$\mr (G)=\min \{\rk A; A \in S(G)\},$$ 
is equivalent to finding the maximal multiplicity of an eigenvalue of $A \in S(G)$. 
\emph{The minimum rank problem} has been resolved for several families of graphs. 
We refer the reader to an excellent survey paper on the problem \cite{MR2350678} where additional references can be found.  A more recent survey paper \cite{2011arXiv1102.5142F} not only gives an up-to-date on the minimum rank problem, but it also talks about several variants to the minimum rank problem that can be found in the literature. 
For example,  the possible inertia of matrices $A \in S(G)$  has been studied in \cite{MR2547901,MR2781688,MR2596445} and the minimum number of distinct eigenvalues in \cite{2013arXiv1304.1205A}.

In this paper we introduce the question of characterizing connected graphs $G$ for which there exists $A \in S(G)$ whose multiplicities of all the eigenvalues are even. For a real symmetric matrix $A$ with the characteristic polynomial $p(x)=\det(xI_n-A)$  the following statements are clearly equivalent:
\begin{enumerate}
\item The multiplicities of all the eigenvalues of $A$ are even.
\item  $p(x)=g(x)^2$, $g(x) \in \R[x]$.
\item $p(x) \geq 0$ for all $x \in \R$. 
\end{enumerate}

\begin{definition}
If a real symmetric matrix $A$ satisfies any (and then all) of the  equivalent conditions above, we will say that $A$ \emph{has the characteristic polynomial a square}.
\end{definition}

\begin{definition}
A graph $G$ is said to \emph{allow the characteristic polynomial a square}, if there exists a matrix $A \in S(G)$ with the characteristic polynomial a square.
\end{definition}

If a graph $G$ on $n$ vertices allows the characteristic polynomial a square, then clearly $\mr(G) \leq n-2.$ However, we will see in Section \ref{Not square} in the paper, that this condition is not sufficient.  In the problem of characterising graphs that allow the characteristic polynomial a square, unlike in the minimum rank problem, we need to keep track of the multiplicities of all the eigenvalues. While the same is true of the inverse eigenvalue problem for graphs, the problem discussed in the present paper is more accessible, while it still exposes some of the difficulties in doing this. Some of the variations on the minimum rank problem were looked at primarily for trees, but as we will see in Section \ref{Not square}, it is not difficult to prove that trees do not allow the characteristic polynomial a square.

The answer to the question for $2 \times 2$ matrices is straightforward. The characteristic polynomial of matrix
 $$A=\npmatrix{a & b \\ b & c }$$ is equal to 
  $p(x)=(x-a)(x-c)-b^2$,
so $p(x) \geq 0$ for all $x \in \mathbb{R}$ if and only if $b=0$ and $a=c$.  Therefore, $A$ has the characteristic polynomial a square  if and only if it is of the form
 $A=a I$  for some real number $a$.

\bigskip

 Our paper is organised as follows.
  In Section \ref{Not square} we introduce some families of graphs that do not allow the characteristic polynomial a square. 
   In Section \ref{Vector space} we identify a vector space of symmetric matrices with the characteristic polynomial a square. In Section \ref{Join} we study constructions that allow us to join two symmetric matrices in such a way that we are able to keep track of the eigenvalues of the resulting matrix. (A simple example of this concept is the tensor product of matrices.) While some of the ideas developed in this section can be extended to other questions related to the possible spectra of matrices is $S(G)$, we concentrate on the applications to the specific problem of this paper. In Section \ref{Rank 2} we characterise all graphs $G$ that allow the characteristic polynomial to be a square and such that there exists $A \in S(G)$ with rank $2$. We conclude our paper with the complete solution to the problem for graphs on $4$ vertices. 

\bigskip

Before proceeding, we introduce some notation needed throughout the paper. 

By $M_n(\R)$ we denote the set of all $nÊ\times n$ matrices with real entries. Let $I_n \in M_n(\R)$ denote the identity matrix and let $E_{ij}$ denote the $0$--$1$ matrix with the only nonzero element in the $(i,j)$-th position.  
For $\mc R, \mc C \subseteq \{1,2,\ldots,n\}$ we denote by $A(\mc R| \mc C)$ the submatrix of $A$ consisting of rows $\mc R$ and columns $\mc C$.

 For a graph $G$ its complement  $G^C$ is the graph on vertices $V(G)$ such that two vertices are adjacent in $G^C$ if and only if they are not adjacent in $G$. 
 The join $G \vee H$ of $G$ and $H$ is the graph union $G \cup H$ together with all the possible edges joining the vertices in $G$ to the vertices in $H$.
 The complete graph on $n$ vertices  will be denoted by $K_n$ and complete bipartite graph on disjoint sets of cardinality $m$ and $n$ by $K_{m,n}$.

\section{Graphs that do not allow characteristic polynomial a square}\label{Not square}

It was observed for example by Johnson and Leal Duarte in \cite[Remark 4]{MR1902112} that  the largest and
the smallest eigenvalue of any matrix $A$, such that $G(A)$ is a tree, each have multiplicity 1. Therefore we have 
the following proposition, that completely solves the problem for trees. 

\begin{proposition}\label{tree}
A tree does not allow the characteristic polynomial a square.
\end{proposition}

In \cite{2013arXiv1304.1205A} the problem of finding the minimum number of distinct eigenvalues $q(G)$ of a graph $G$ is discussed, and the following theorem \cite[Theorem 3.2]{2013arXiv1304.1205A} is proved. 

\begin{theorem}[\cite{2013arXiv1304.1205A}]\label{Fallat}
If there are vertices $u$, $v$ in a connected graph $G$ at distance $d$ and the path of length $d$ from $u$ to $v$ is unique, then $q(G) \geq d+1$.
\end{theorem}

In order for the graph $G$ on $2n$ vertices to allow the characteristic polynomial a square we need $q(G) \leq n.$  This gives us an immediate corollary of Theorem \ref{Fallat}.

\begin{corollary}\label{FallatCor}
Let $G$ be a connected graph on $2n$ vertices. 
If there are vertices $u$, $v$ in a connected graph $G$ at distance $d$, $d \geq n,$ and the path of length $d$ from $u$ to $v$ is unique, then $G$ does not allow the characteristic polynomial to be a square.
\end{corollary}

\begin{example}
By $G \cup_v P_m \cup_u H$ we denote the graph on $|G| +|H|+ m-2$ vertices constructed as a union of graphs $G$, $H$ and  a path $P_m$ on $m$ vertices, 
 where vertex $v$ of degree $1$ in $P_m$ is identified with vertex $v$ in $G$ and the other vertex $u$ of degree $1$ in $P_m$ is identified with  vertex $u$ in $H$. 
 
Let $v_1$ be a neighbour of $v$ in $G$ and let $u_1$ be a neighbour of $u$ in $H$. Then the distance from $v_1$ to $u_1$ in $G \cup_v P_m \cup_u H$ is $m+1$, and the path between $u_1$ and $v_1$ of length $m+1$ is unique. 
By Corollary \ref{FallatCor} we conclude that if $|G \cup_v P_m \cup_u H|=2n$, then for $m\geq n-1$ the graph $G \cup_v P_m \cup_u H$ does not allow the characteristic polynomial to be a square. 

If $H=P_2$, then $G \cup_v P_{m}=G \cup_v P_{m-1} \cup_u H$. By the same argument as before we have that $G \cup_v P_{m}$ does not allow the characteristic polynomial a square if $m \geq n$.
\end{example}

\begin{example}\label{cycle+1}
In the case  $C_k \cup_v P_{m}$, $2n=k+m-1$, there exists a unique path of length $m-1+\left\lfloor \frac{k-1}{2} \right\rfloor$ from vertex $u$ of degree 1 on $P_m$ to vertex $v_1$ at distance $\left\lfloor \frac{k-1}{2} \right\rfloor$ from $v$ on the cycle. 
Note that condition  $m-1+\left\lfloor \frac{k-1}{2} \right\rfloor \geq n$ is equivalent to  $m\geq 2$. So, if $m \geq 2$, then $C_k \cup_v P_{m}$ does not allow the characteristic polynomial a square.
\end{example}

We can easily find examples of graphs $G$, $|G|=2n$, with $q(G) \leq n $, that do not allow the characteristic polynomial to be a square. For example, let $G$ be a star with diameter $2$ on six vertices. Then a matrix
 $$A=\left(
\begin{array}{cccccc}
 0 & 1 & 1 & 1 & 1 & 1 \\
 1 & 0 & 0 & 0 & 0 & 0 \\
 1 & 0 & 0 & 0 & 0 & 0 \\
 1 & 0 & 0 & 0 & 0 & 0 \\
 1 & 0 & 0 & 0 & 0 & 0 \\
 1 & 0 & 0 & 0 & 0 & 0 \\
\end{array}
\right)$$
is in $S(G).$ $A$ has only three distinct eigenvalues: $0,$ $\sqrt{5}$, $-\sqrt{5}$, however, their multiplicities are $4,$ $1$, $1$, and since $G$ is a tree, we already know that it does not allow the characterstic polynomial to be a square.  

\bigskip

Moreover, there exist graphs  not covered in Proposition \ref{tree} and Corollary \ref{Fallat} that do not allow the characteristic polynomial to be a square. 

\begin{theorem}
Matrix of the form
\begin{equation}\label{structure1} 
 A=\npmatrix{d & b^T & c^T \\
                          b & B & 0 \\
                          c & 0 & D} \in \R^{2n \times 2n},
\end{equation}                          
where $d \in \R$, $B \in \R^{2 \times 2}$ and $D \in \R^{(2n-3) \times (2n-3)}$ diagonal matrix and $c$ is a vector that has all its components nonzero, cannot have characteristic polynomial a perfect square. 
\end{theorem}

\begin{proof}
Let us assume that the matrix
$$
 A=\npmatrix{d & b^T & c^T \\
                          b & B & 0 \\
                          c & 0 & D} 
$$       
 has eigenvalues $\lambda_1,\lambda_2,\ldots,\lambda_t$, with multiplicities $2k_1,2k_2,\ldots,2k_t$  and let us denote
  $$A_1=\npmatrix{d & b^T \\ b & B}.$$ Note that if $D-\lambda_iI$ has all its entries nonzero for some $\lambda_i$, then 
 $$2n-2k_i=\rk (A-\lambda_iI) \geq \rk (A_1-\lambda_iI)+\rk (D-\lambda_iI)=\rk (A_1-\lambda_iI)+2n-3,$$
 and so $\rk(A_1-\lambda_iI) \leq 1$. It follows that $D-\lambda_iI$ has all its entries nonzero for at most one eigenvalue, say $\lambda_1$,
and $D-\lambda_iI$ has
  at least one zero element for $i=2,\ldots,t$
Now we have $ \rk (A-\lambda_1 I) \geq \rk(A_1-\lambda_1I)+\rk(D-\lambda_1I)$ and $ \rk (A-\lambda_i I) \geq \rk(A_1-\lambda_iI)+\rk(D-\lambda_iI)+2$ for  $i\geq 2$. Since:
\begin{align*}
\sum_{i=1}^t \rk( A_1-\lambda_iI) &\geq 3t-3\\
\sum_{i=1}^t \rk( D-\lambda_iI) &\geq (2n-3)(t-1)
\end{align*}
we get:
$$2nt-2n=\sum_{i=1}^t\rk(A-\lambda_iI) \geq 3t-3+(2n-3)(t-1)+2(t-1)$$
which gives us $t \leq 1$, and this is clearly impossible. 
\end{proof}

\section{Vector space of symmetric matrices that have the characteristic polynomial a square}\label{Vector space}

A vector space of $2n \times 2n$ upper triangular matrices with the diagonal of the form $d_1, \ldots, d_n, d_1, \ldots,d_n$ is an obvious example of a vector space of matrices whose characteristic polynomial is a square. The dimension 
of this space is $\frac{n^2}{2}$. In this section we give an example of a vector space of  $2n \times 2n$ symmetric matrices with the same property. 

\begin{theorem}\label{AS}
Let ${\mathcal V}$ be a set of symmetric matrices %over a general field $\F$ 
of the form
$$\npmatrix{A & S \\ -S & A},$$
such that $A  \in M_n(\R)$ is a symmetric matrix, and  $S  \in M_n(\R)$ a skew symmetric matrix. Then ${\mathcal V}$ is a  vector space of dimension $\frac{n^2}{2}$ in which   the characteristic polynomial of every matrix
is a square.
 \end{theorem}

\begin{proof}
Let 
 $$T=\npmatrix{I_n & i I_n \\ 0 & I_n} \text{ and } M=\npmatrix{A & S \\ -S & A}.$$
 Then 
\begin{align*}
  TMT^{-1}&=\npmatrix{I_n & i I_n \\ 0 & I_n}\npmatrix{A & S \\ -S & A}\npmatrix{I_n & -i I_n \\ 0 & I_n}\\
  &=\npmatrix{A+i S & 0 \\ -S & A-i S}.
\end{align*} 
Since $(A+i S)^T=A-i S$, the spectrum of $A+i S$ is equal to the spectrum of $A-i S.$ 
\end{proof}

\begin{corollary}\label{cycle}
 The cycle on $2n$ vertices allows the characteristic polynomial a square. 
\end{corollary}
\begin{proof}
In Theorem \ref{AS} we take $$S=E_{n1}-E_{1n} \in M_n(\R) \text{ and } A = \sum\limits_{|i-j|=1} E_{i,j}\in  M_n(\R)$$ 
to achieve the pattern corresponding to a cycle.
\end{proof}

In addition to cycles,  Theorem \ref{AS} covers a larger set of patterns that allow the characteristic polynomial to be a square. However,
as we will see in later sections, there are other patterns.
Below we illustrate our Theorem with a concrete example.

\begin{example}
The only graph on $4$ vertices  corresponding to matrices in Theorem \ref{AS} is a cycle.
All graphs on $6$ vertices that are covered by Theorem \ref{AS} are
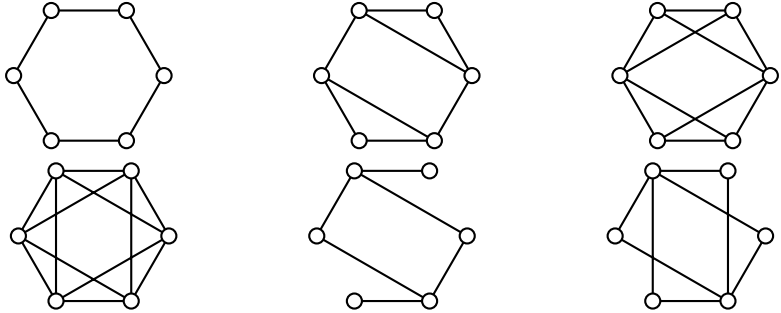
\begin{center}
   \begin{tikzpicture}[style=thick, scale=1]
		\draw \foreach \x in {0,60,120,180,240,300} {
				(\x:1) node{} -- (\x+60:1) node{}
		};		
		\draw[fill=white] \foreach \x in {0,60,120,180,240,300} {
		                    (\x:1) circle (1mm)
		};		
   \end{tikzpicture} \qquad \qquad
    \begin{tikzpicture}[style=thick, scale=1]
		\draw \foreach \x in {0,60,120,180,240,300} {
				(\x:1) node{} -- (\x+60:1) node{}
		};		
		\draw (0:1) node{} -- (120:1) node{};
		\draw (180:1) node{} -- (300:1) node{};
		\draw[fill=white] \foreach \x in {0,60,120,180,240,300} {
		                    (\x:1) circle (1mm)
		};		
   \end{tikzpicture}\qquad \qquad
    \begin{tikzpicture}[style=thick, scale=1]
		\draw \foreach \x in {0,60,120,180,240,300} {
				(\x:1) node{} -- (\x+60:1) node{}
		};		
		\draw (240:1) node{} -- (0:1) node{} -- (120:1) node{};
		\draw (60:1) node{} -- (180:1) node{} -- (300:1) node{};
		\draw[fill=white] \foreach \x in {0,60,120,180,240,300} {
		                    (\x:1) circle (1mm)
		};		
   \end{tikzpicture}
\end{center}

\begin{center}
    \begin{tikzpicture}[style=thick, scale=1]
		\draw \foreach \x in {0,60,120,180,240,300} {
				(\x:1) node{} -- (\x+60:1) node{}
				(\x:1) node{} -- (\x+120:1) node{}				
		};		
		\draw[fill=white] \foreach \x in {0,60,120,180,240,300} {
		                    (\x:1) circle (1mm)
		};		
   \end{tikzpicture}\qquad \qquad
    \begin{tikzpicture}[style=thick, scale=1]
		\draw \foreach \x in {60,120,240,300} {
				(\x:1) node{} -- (\x+60:1) node{}
		};		
		\draw (0:1) node{} -- (120:1) node{};
		\draw (180:1) node{} -- (300:1) node{};
		\draw[fill=white] \foreach \x in {0,60,120,180,240,300} {
		                    (\x:1) circle (1mm)
		};		
   \end{tikzpicture}\qquad \qquad
    \begin{tikzpicture}[style=thick, scale=1]
		\draw \foreach \x in {0,120,180,300} {
				(\x:1) node{} -- (\x+120:1) node{}
		};		
		\draw \foreach \x in {60,120,240,300} {
				(\x:1) node{} -- (\x+60:1) node{}
		};		
		\draw[fill=white] \foreach \x in {0,60,120,180,240,300} {
		                    (\x:1) circle (1mm)
		};		
   \end{tikzpicture}
\end{center}  
\end{example}

\bigskip

\section{Joining two graphs}\label{Join}

It is well known that tensor product $A \otimes B$ has eigenvalues 
$\lambda_i \mu_j$, $i=1,\ldots,n$, $j=1,\ldots,m$, where 
$\lambda_1,\ldots,\lambda_n$ are the eigenvalues of $A$ and $\mu,\ldots,\mu_m$ are the eigenvalues of $B$. We conclude that if $A$ has the characteristic polynomial a square, so does $A \otimes B.$ 

Tensor product of two graphs is usually defined in the following way:
$G \times H=(V(G \times H), E(G \times H))$. where $V(G \times H)=V(G) \times V(H)$
and $((u,u'),(v,v')) \in E(G \times H)$ if and only if $ (u,v) \in E(G)$ and $(u',v') \in E(H)$.
The tensor product of two graphs corresponds to tensor product of matrices in the way that the adjacency matrix of a tensor product of graphs is
the tensor product of adjacency matrices. 

Pattern of adjacency matrix of a graph $G$ agrees with the pattern of any matrix $A \in S(G)$, with the pattern of the diagonal not specified. Note that the pattern of the diagonal entries of $A$ and $B$ affects the off-diagonal entries of their tensor product $A \otimes B$. 

 If $A \in S(G)$ has the characteristic polynomial a square then the same is true for $A+\alpha I$ for all $\alpha \in \R$. This observation allows us to assume that all the diagonal entries of $A \in S(G)$ are different from  zero. Different graphs $G$ allow different patterns on the diagonal. (For example, by adding a scalar to the matrix, we can always achieve at least one diagonal entry to be equal to zero). To avoid considering each case separately we define $(u,u) \in E(G)$ for all $u \in V(G)$, while we can keep the pattern of the diagonal elements of matrices in $S(H)$ free.  
 This gives us the 
following theorem.

\begin{theorem}
 Let $G$ be a graph that allows the characteristic polynomial a square, where we define $(u, u) \in E(G)$. Then the tensor product $G \times H$  allows the characteristic polynomial a square for any graph $H$.
\end{theorem}

\begin{example}
 If we take a $4 \times 4$ matrix $C \in  S(C_4)$ as defined in the proof of Corollary \ref{cycle} and $2 \times 2$ matrices $A$ and $B$ with zero-nonzero patterns $\npmatrix{0 & * \\
               * & *}$ and $\npmatrix{* & * \\
               * & *}$,
        the tensor products $A \otimes C$ and $B \otimes C$ have the characteristic polynomial a square. Thus, the following two graphs
        
   
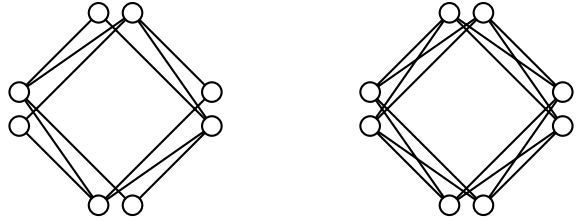
\begin{center}
           \begin{tikzpicture}[style=thick, scale=1.3]
		\draw \foreach \x in {10,100,190,280} {
				(\x:1) node{} -- (\x+70:1) node{} -- (\x+160:1) node{}
		};		
		\draw \foreach \x in {80,170,260,350} {
				(\x:1) node{} -- (\x+110:1) node{}
		};		
		\draw[fill=white] \foreach \x in {10,80,100,170,190,260,280,350} {
		                    (\x:1) circle (1mm)
		};		
	\end{tikzpicture}
	\qquad
	\qquad
           \begin{tikzpicture}[style=thick, scale=1.3]
		\draw \foreach \x in {10,100,190,280} {
				(\x:1) node{} -- (\x+70:1) node{} -- (\x+160:1) node{}
				(\x:1) node{} -- (\x+90:1) node{}
		};		
		\draw \foreach \x in {80,170,260,350} {
				(\x:1) node{} -- (\x+110:1) node{}
		};		
		\draw[fill=white] \foreach \x in {10,80,100,170,190,260,280,350} {
		                    (\x:1) circle (1mm)
		};		
	\end{tikzpicture}
   \end{center}     
         allow the characteristic polynomial a square.

\end{example}

Next, we investigate the applications of  Lemma 5 in \cite{MR2098598}, that we give below,  to our problem.

\begin{lemma}[\cite{MR2098598}]\label{HS04}
Let $B$ be a symmetric $m \times m$ matrix with eigenvalues $\mu_1, \mu_2,\ldots,\mu_m$, and let $u$ be an eigenvector corresponding to $\mu_1$ normalized so that $u^Tu=1$. Let $A$ be an $n \times n$ symmetric matrix with a diagonal element $\mu_1$
 \begin{equation}\label{c1f1}
  A=\npmatrix{A_1 & b \\
               b^T & \mu_1}
 \end{equation}
 and eigenvalues $\lambda_1, \ldots, \lambda_n.$
 Then the matrix 
  \[C=\npmatrix{A_1 & bu^T \\
               ub^T & B}\]
               has eigenvalues $\lambda_1,\ldots,\lambda_n,\mu_2,\ldots, \mu_m$. 
\end{lemma}

\begin{remark}\label{rem:eigenvectors}
In \cite{MR2098598} the eigenvectors of $C$ in terms of the eigenvectors of $A$ and $B$ are given in the following way. Let 
 $$\npmatrix{v_i \\ \alpha_i}, \, u_i \in \R^{n-1}, \, \alpha_i \in \R$$
 be the orthonormal set of eigenvectors of $A$ corresponding to $\lambda_i,$ $i=1,2,
 \ldots,n$, and let $u_i,$ $i=2,\ldots,m,$ together with $u$ be the orthonormal set of eigenvectors corresponding to $\mu_i$. Then $C$ has the following eigenvectors: eigenvector of $C$ corresponding to  eigenvalue $\lambda_i,$ $i=1,2,\ldots,n$, is equal to:
  $$\npmatrix{v_i \\ \alpha_i u},$$
  and the eigenvector of $C$ corresponding to $\mu_i,$ $i=2,\ldots,m$, is equal to:
  $$\npmatrix{0 \\ u_i}.$$    
\end{remark}

\begin{theorem}\label{thm:2m+1}
Let $G$ be a graph that can be realized by a matrix whose characteristic polynomial is a square. Let $v$ be a vertex in $G$. Let $G_{2m+1}$ be a graph constructed from $G$ in the following way: 
vertex $v$ is replaced by a clique $K_{2m+1}$, and every vertex in $K_{2m+1}$ has the same neighbours in the rest of the graph as $v$ has in $G$. Then $G_{2m+1}$ can be realized by a matrix whose 
characteristic polynomial is a square.
\end{theorem}

\begin{proof}
Let $A$ be a matrix with characteristic polynomial $p(x)^2$ that realizes $G$. Since we can add a scalar to $A$ and multiply it by a nonzero scalar, we can, without loss of generality assume that the diagonal element of $A$ corresponding to vertex $v$ is equal to $2m+1$:
 \begin{equation}
  A=\npmatrix{A_1 & b \\
               b^T & 2m+1}.
 \end{equation}
Let $B$ be a $(2m+1)\times (2m+1)$ matrix with all its elements equal to $1$. Then $2m+1$ is an eigenvalue of $B$ with corresponding eigenvector $e=(1,1,\ldots, 1)^T \in \R^{2m+1}$  and all other eigenvalues of $B$
 are equal to 0.
Lemma \ref{HS04} tells us that the characteristic polynomial of
  \[C=\npmatrix{A_1 & be^T \\
               eb^T & B}\]
is equal to $x^{2m}p(x)^2$. 
\end{proof}
\begin{example}
 Theorem \ref{thm:2m+1} implies that the following graphs on $6$ vertices allow the characteristic polynomial a square: $K_6$, $K_2^C$, $(K_{3,1} \cup 2 K_1)^C$ and $(K_{3,1} \cup K_{1,1})^C$.
\end{example}

\begin{lemma}
Let $A$ be a symmetric matrix of the form
$$A=\npmatrix{A_1 & a \\ a^T & 2}$$
with the characteristic polynomial $p(x)$, and let
 $B$ be a symmetric matrix of the form:
 $$B=\npmatrix{B_1 & b \\ b^T & 0}$$ with the characteristic polynomial $q(x)$. Then the matrix
$$C=\npmatrix{B_1 & 0 & \frac{\sqrt{2}}{2} b & -\frac{\sqrt{2}}{2}b \\ 0 & A_1 & \frac{\sqrt{2}}{2}a & \frac{\sqrt{2}}{2}a \\ \frac{\sqrt{2}}{2}b^T & \frac{\sqrt{2}}{2}a^T & 1 & 1 \\ -\frac{\sqrt{2}}{2}b^T & \frac{\sqrt{2}}{2}a^T  & 1 & 1}$$
has the characteristic polynomial $p(x)q(x).$
\end{lemma}

\begin{proof}
Using Lemma \ref{HS04} we first join $A$ with 
 $$\npmatrix{1 & 1 \\ 1 & 1}$$
to obtain matrix 
$$D=\npmatrix{A_1 & \frac{\sqrt{2}}{2}a & \frac{\sqrt{2}}{2}a \\ \frac{\sqrt{2}}{2}a^T  & 1 & 1 \\ \frac{\sqrt{2}}{2} a^T & 1 & 1}$$
with the characterisic polynomial $p(x)x$. Notice that eigenvector of $D$ corresponding to $0$ is equal to $( 0 , \frac{\sqrt{2}}{2}, -\frac{\sqrt{2}}{2})^T$. Now we use Lemma \ref{HS04} again to join $D$ with $B$ to obtain matrix $C$ with the characterisitc polynomial $p(x)q(x).$
\end{proof}

\begin{corollary}\label{thm:6}
Let $G_A$ and $G_B$ be  graphs that allow characteristic polynomial a square. Let $v_A$ be a vertex in $G_A$ and let $v_B$ be a vertex in $G_B$. Let $G_C$ be a graph constructed by adding the following edges to $G_A \cup G_B$: vertices $v_A$ and $v_B$ are joined by an edge, $v_A$ has the same edges to the rest of $G_B$ as $v_B$ and $v_B$ has the same edges to the rest of $G_A$ as $v_A$. Then $G_C$
allows the characteristic polynomial a square.
\end{corollary}

\begin{example}
  If $G_A=G_B=C_4$, then graph $G_C$:
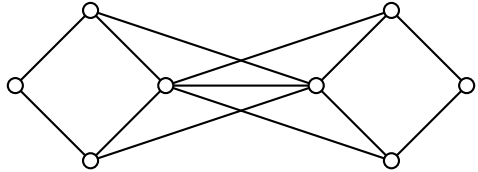
  \begin{center}
    \begin{tikzpicture}[style=thick]
	\draw (-1,0) -- (-2,1) -- (-3,0) -- (-2,-1) -- (-1,0) -- (1,0) -- (2,1) -- (3,0) -- (2,-1) -- (1,0) -- (-2,1);
	\draw (2,1) -- (-1,0) -- (2,-1);
	\draw (-2,-1) -- (1,0);
	\draw[fill=white] (-1,0) circle (1mm) (-2,-1) circle (1mm) (-2,1) circle (1mm) (-3,0) circle (1mm) (1,0) circle (1mm) (2,-1) circle (1mm) (2,1) circle (1mm) (3,0) circle (1mm) ;
\end{tikzpicture}  
\end{center}
 allows the characteristic polynomial a square.
\end{example}

Using Lemma \ref{HS04} we can obtain further examples that are not covered by Theorem \ref{thm:2m+1} or by Corollary \ref{thm:6}. Namely, we will prove that graphs:
 
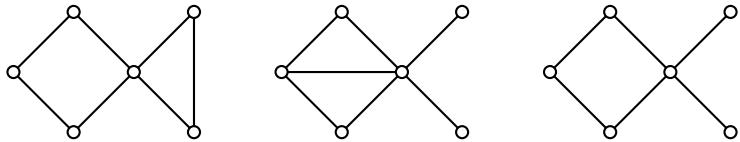
\begin{center}
    \begin{tikzpicture}[style=thick,scale=0.8]
	\draw (-2,0) -- (-1,1) -- (0,0) -- (1,1) -- (1,-1) -- (0,0) -- (-1,-1) -- (-2,0);
	\draw[fill=white] (-2,0) circle (1mm) (-1,-1) circle (1mm) (-1,1) circle (1mm) (0,0) circle (1mm) (1,1) circle (1mm) (1,-1) circle (1mm);
\end{tikzpicture}  
\qquad
    \begin{tikzpicture}[style=thick,scale=0.8]
	\draw (-2,0) -- (-1,1) -- (0,0) -- (1,1);
	\draw (1,-1) -- (0,0) -- (-1,-1) -- (-2,0) -- (0,0);
	\draw[fill=white] (-2,0) circle (1mm) (-1,-1) circle (1mm) (-1,1) circle (1mm) (0,0) circle (1mm) (1,1) circle (1mm) (1,-1) circle (1mm);
\end{tikzpicture}  
\qquad
    \begin{tikzpicture}[style=thick,scale=0.8]
	\draw (-2,0) -- (-1,1) -- (0,0) -- (1,1);
	\draw (1,-1) -- (0,0) -- (-1,-1) -- (-2,0);
	\draw[fill=white] (-2,0) circle (1mm) (-1,-1) circle (1mm) (-1,1) circle (1mm) (0,0) circle (1mm) (1,1) circle (1mm) (1,-1) circle (1mm);
\end{tikzpicture}  
\end{center}
allow the characteristic polynomial a square. They depend on knowing appropriate eigenvector of one of the matrices that we are using, so they are difficult to extend to general dimensions.  

\begin{example}
Let 
$$A_1=\left(
\begin{array}{ccc}
 0 & 1 & 0 \\
 1 & 1 & -1 \\
 0 & -1 & 0
\end{array}
\right), \, A_3=\left(
\begin{array}{ccc}
 1 & 1 & 0 \\
 1 & 1/2 & 1 \\
 0 & 1 & 0
\end{array}
\right), \, A_4=\npmatrix{3/4 & 1/4 \\ 1/4 & 3/4}.$$
The matrix $A_1$ has eigenvalues $\{2,-1,0\}$ and $(1/\sqrt{2})(1,0,1)^T$ is the normalized eigenvector for $A_1$ corresponding to $0$. $A_3$ has eigenvalues $\{2,-1,1/2\}$ and $A_4$ has eigenvalues $\{1,1/2\}$ with the normalized eigenvector $(1/\sqrt{2})(1,1)^T$ corresponding to $1$. 

First we apply the construction described in Lemma \ref{HS04} to $A_3$ and $A_4$ connecting them through the diagonal element $1$ of $A_3$ and the eigenvalue $1$ of $A_4$. This gives us: 
$$A_2=\left(
\begin{array}{cccc}
 \frac{3}{4} & \frac{1}{4} & \frac{1}{\sqrt{2}} & 0 \\
 \frac{1}{4} & \frac{3}{4} & \frac{1}{\sqrt{2}} & 0 \\
 \frac{1}{\sqrt{2}} & \frac{1}{\sqrt{2}} & \frac{1}{2} & 1 \\
 0 & 0 & 1 & 0
\end{array}
\right)$$
with eigenvalues $\{2,-1,1/2,1/2\}$. Next we join $A_2$ and $A_1$ using Lemma \ref{HS04} through the diagonal element $0$ of $A_2$ and eigenvalue $0$ of $A_1$. 
This gives us 
$$A=\left(
\begin{array}{cccccc}
 \frac{3}{4} & \frac{1}{4} & \frac{1}{\sqrt{2}} & 0 & 0 & 0 \\
 \frac{1}{4} & \frac{3}{4} & \frac{1}{\sqrt{2}} & 0 & 0 & 0 \\
 \frac{1}{\sqrt{2}} & \frac{1}{\sqrt{2}} & \frac{1}{2} & \frac{1}{\sqrt{2}} & 0 &
   \frac{1}{\sqrt{2}} \\
 0 & \frac{1}{\sqrt{2}} & 0 & 0 & 1 & 0 \\
 0 & 0 & 0 & 1 & 1 & -1 \\
 0 & \frac{1}{\sqrt{2}} & 0 & 0 & -1 & 0
\end{array}
\right)$$
with eigenvalues $\{2,2,-1,-1,1/2,1/2\}$. We have constructed a matrix $A$ with the characteristic polynomial a square. 
\end{example}

\begin{example}
Let 
$$A_1=\left(
\begin{array}{ccc}
 -1 & 1 & 0 \\
 1 & 0 & -1 \\
 0 & -1 & 1
\end{array}
\right), \, A_3=\left(
\begin{array}{ccc}
 0 & 1 & 0 \\
 1 & 0 & \sqrt{2} \\
 0 & \sqrt{2} & 0
\end{array}
\right), \, A_4=\npmatrix{0 & 0 \\ 0 & 0}.$$
The matrix $A_1$ has eigenvalues $\{\sqrt{3},-\sqrt{3},0\}$ and $(1/\sqrt{3})(1,1,1)^T$ is the normalized eigenvector for $A_1$ corresponding to $0$. $A_3$ has eigenvalues $\{\sqrt{3},-\sqrt{3},0\}$ and $A_4$ has eigenvalues $\{0,0\}$ with a normalized eigenvector $(1/\sqrt{2})(1,1)^T$ corresponding to $0$. 

As in the previous example, we first use Lemma \ref{HS04} to join $A_3$ and $A_4$ through a diagonal element $0$ of $A_3$ and an eigenvalue $0$ of $A_4$ producing:
 $$A_2=\left(
\begin{array}{cccc}
 0 & 0 & \frac{1}{\sqrt{2}} & 0 \\
 0 & 0 & \frac{1}{\sqrt{2}} & 0 \\
 \frac{1}{\sqrt{2}} & \frac{1}{\sqrt{2}} & 0 & \sqrt{2} \\
 0 & 0 & \sqrt{2} & 0
\end{array}
\right)$$
with eigenvalues $\{\sqrt{3},-\sqrt{3},0,0\}.$
Combining $A_2$ and $A_1$ through the eigenvalue $0$ of $A_1$ and the diagonal element $0$ of $A_2$ gives us:
$$A=\left(
\begin{array}{cccccc}
 0 & 0 & \frac{1}{\sqrt{2}} & 0 & 0 & 0 \\
 0 & 0 & \frac{1}{\sqrt{2}} & 0 & 0 & 0 \\
 \frac{1}{\sqrt{2}} & \frac{1}{\sqrt{2}} & 0 & \sqrt{\frac{2}{3}} & \sqrt{\frac{2}{3}}
   & \sqrt{\frac{2}{3}} \\
 0 & 0 & \sqrt{\frac{2}{3}} & -1 & 1 & 0 \\
 0 & 0 & \sqrt{\frac{2}{3}} & 1 & 0 & -1 \\
 0 & 0 & \sqrt{\frac{2}{3}} & 0 & -1 & 1
\end{array}
\right)$$
with eigenvalues $\{\sqrt{3},\sqrt{3},-\sqrt{3},-\sqrt{3},0,0\}.$
\end{example}

\begin{example}
Let 
$$A_1=A_3=\left(
\begin{array}{ccc}
 0 & 1 & 0 \\
 1 & 1 & -1 \\
 0 & -1 & 0
\end{array}
\right), \, A_4=\npmatrix{0 & 0 \\ 0 & 0}.$$
Combining $A_3$ with $A_4$ gives us
$$A_2=\left(
\begin{array}{cccc}
 0 & 0 & \frac{1}{\sqrt{2}} & 0 \\
 0 & 0 & \frac{1}{\sqrt{2}} & 0 \\
 \frac{1}{\sqrt{2}} & \frac{1}{\sqrt{2}} & 1 & -1 \\
 0 & 0 & -1 & 0
\end{array}
\right)$$
with eigenvalues $\{2,-1,0,0\}$ and combining $A_2$ with $A_1$ produces
$$A=\left(
\begin{array}{cccccc}
 0 & 0 & \frac{1}{\sqrt{2}} & 0 & 0 & 0 \\
 0 & 0 & \frac{1}{\sqrt{2}} & 0 & 0 & 0 \\
 \frac{1}{\sqrt{2}} & \frac{1}{\sqrt{2}} & 1 & -\frac{1}{\sqrt{2}} & 0 &
   -\frac{1}{\sqrt{2}} \\
 0 & 0 & -\frac{1}{\sqrt{2}} & 0 & 1 & 0 \\
 0 & 0 & 0 & 1 & 1 & -1 \\
 0 & 0 & -\frac{1}{\sqrt{2}} & 0 & -1 & 0
\end{array}
\right)$$
with eigenvalues $(2,2,-1,-1,0,0).$
\end{example}

Lemma \ref{HS04} can be used to prove the following Theorem that we will need in the proof of some of the results later on. 

\begin{theorem}\label{thm:Knall}
For any given list of real numbers $\sigma=(\lambda_1,\lambda_2,\ldots,\lambda_n)$, $\lambda_1 \neq \lambda_2$, there exists $A_n \in S(K_n)$ with the spectrum $\sigma$.

Furthermore, given any zero-nonzero pattern of a vector in $\R^n$ that contains at least two nonzero elements, $A_n$ can be chosen in such a way that there exist an eigenvector corresponding to $\lambda_1$ with the given pattern.
\end{theorem}

\begin{proof}
We will use induction on $n$. Without loss of generality we may assume that $\lambda_1>\lambda_2$. For $n=2$, we define the matrix
 $$A_2=\npmatrix{\lambda_1+\lambda_2-d_2 & \sqrt{-(\lambda_1-d_2)(\lambda_2-d_2)} \\ \sqrt{-(\lambda_1-d_2)(\lambda_2-d_2)}&d_2}.$$
If $\lambda_1>d_2>\lambda_2$, then $A_2$ is well defined,  $A_2 \in S(K_2)$, the spectrum  of $A_2$ is $\{\lambda_1,\lambda_2\}$ 
and the eigenvectors corresponding to $\lambda_1$ and to $\lambda_2$ do not have any zero entries.

Now we assume that the theorem holds for $n-1$. Since $S(K_n)$ is closed under the permutational similarity, we need to show that we can find a matrix $A_n \in S(K_n)$ with eigenvalues $\{\lambda_1,\lambda_2,\ldots,\lambda_n\}$ and with an eigenvector corresponding to $\lambda_1$ with the given number of nonzero elements $k$, $2 \leq k \leq n$. 

First we assume that $k \geq 3.$ 
Let $A_{n-1} \in S(K_{n-1})$ have the spectrum $\{\lambda_1,\lambda_2,\ldots,\lambda_{n-1}\}$ and an eigenvector corresponding to $\lambda_1$ with $k-1$ nonzero elements. Using permutational similarity, we can assume that the $(n-1)$-st element of this eigenvector is different from $0$. Furthermore, we assume that the last diagonal element of $A_{n-1}$ is different from $\lambda_n$. We denote this diagonal element by $d_{n-1}$. We can assure that such a diagonal element exists by appropriately choosing a diagonal element $d_2$,  $\lambda_1>d_2>\lambda_2$, at the first step of this construction. By $n=2$ case we can find $B_n \in S(K_2)$ with the spectrum $\{d_{n-1}, \lambda_n\}$ and with the eigenvectors of $B_n$ not containing any zero entries. 
 Now we use Lemma \ref{HS04} to construct a matrix $A_n$ with eigenvalues $\{\lambda_1,\lambda_2,\ldots,\lambda_n\}$.
 By Remark \ref{rem:eigenvectors} the eigenvector of $A_n$ corresponding to $\lambda_1$ will have $k$ nonzero entries.

 Notice that in the construction above the eigenvalue $\lambda_n$ will have a corresponding eigenvector with only two nonzero elements. So to deal with the $k=2$ case, we can use the same construction as described above to construct a symmetric matrix $A_{n-1}$ with the eigenvalues $\{\lambda_1,\lambda_2,\ldots,\lambda_n\}$, except that we add $\lambda_1$ to the spectrum  $\{\lambda_2,\ldots,\lambda_n\}$ of $A_{n-1}.$
\end{proof}

\begin{corollary}\label{thm:Kn}
 All complete graphs $K_{2m}$, $m \geq 2$, allow the characteristic polynomial a square.
\end{corollary}

\begin{corollary}\label{thm:GKn}
 For every graph $G$ on $n-1$ vertices, the join $$G \vee K_{n+1}$$ allows the characteristic polynomial a square.
\end{corollary}

\begin{proof}
 Let us take any symmetric $n \times n$ matrix 
 \begin{equation*}
  A=\npmatrix{A_1 & b \\
               b^T & \mu_1}
 \end{equation*}
 with $b$ having only nonzero entries and $A_1 \in S(G)$. 
 Since $A$ is not a scalar matrix, $A$ has eigenvalues $\lambda_1,\ldots,\lambda_n$, where $\lambda_1 \ne \lambda_2$.
 By Theorem \ref{thm:Knall} there exists an  $(n+1) \times (n+1)$ matrix $B$ having eigenvalues $\mu_1,\lambda_1,\ldots,\lambda_n$ and an eigenvector $u$ with nonzero entries corresponding to $\mu_1$.
 By Lemma \ref{HS04} it follows that matrix $\npmatrix{A_1 & bu^T \\
               ub^T & B}$
               has eigenvalues $\lambda_1,\lambda_1,\ldots,\lambda_n, \lambda_n$ and its graph is a join of $G(A_1)$ and $K_{n+1}$. 
\end{proof}

\begin{remark}
Matrix $A$ in the proof above can have repeated eigenvalues. In this case, $G \vee K_{m}$ allows the characteristic polynomial a square for $m <n+1$. It is easy to find a bound for the minimal $m$ in terms of the minimal rank of $G \vee K_1$.
\end{remark}

To illustrate the scope of our method, we give the example below.

\begin{example}
The matrix 
$$A=\left(
\begin{array}{ccccc}
 0 & 1 & 0 & 1 & 1 \\
 1 & 0 & 1 & 0 & 1 \\
 0 & 1 & 0 & 1 & 1 \\
 1 & 0 & 1 & 0 & 1 \\
 1 & 1 & 1 & 1 & 0 \\
\end{array}
\right)=\npmatrix{A_1 & e \\ e^T & 0}$$
has eigenvalues $\{1+\sqrt{5},1-\sqrt{5},-2,0,0\}.$
Theorem \ref{thm:Knall} tells us that there exists a matrix $B \in K_{2n},$ $n \geq 2$ with eigenvalues:
$$\{0,1+\sqrt{5},1-\sqrt{5},-2,\lambda_1,\lambda_1,\ldots,\lambda_{n-2},\lambda_{n-2}\}$$
and an eigenvector $u$ corresponding to $0$ having any given zero-nonzero pattern as long as it has at least two nonzero entries. The matrix:
 $$C=\npmatrix{A_1 & eu^T \\ ue^T & B}$$
 has eigenvalues:
 $$\{1+\sqrt{5},1+\sqrt{5},1-\sqrt{5},1-\sqrt{5},-2,-2,0,0,\lambda_1,\lambda_1,\ldots,\lambda_{n-2},\lambda_{n-2}\},$$
 hence the characteristic polynomial a square. 
\end{example}

\section{Realisation with rank $2$ matrices}\label{Rank 2}

We denote by ${\rm mr}_+(G)$ the minimum rank among all positive semidefinite symmetric matrices corresponding to $G$.
 A characterisation of graphs with $\mr_+(G) \leq 2$ is given in \cite{MR2111528}. Here answer the question, what graphs that allow $\mr_+(G) \leq 2$ can be realized by a matrix that has the two nonzero eigenvalues equal. 

Below we state  Corollary 1 from  \cite{MR2111528}.

\begin{theorem}[\cite{MR2111528}]\label{thm:BvHL}
Let $G$ be a graph on $n$ vertices. Then ${\rm mr}_+(G)\leq 2$ if and only if $G^c$ has the form  
$$(K_{p_1,q_1}\cup K_{p_2,q_2} \cup \ldots \cup K_{p_k,q_k})\vee K_r$$
 for appropriate nonnegative integers $k, p_1,q_1,\ldots,p_k,q_k,r$ with $p_i+q_i>0,$ $i=1,2,\ldots,k.$
\end{theorem}

Here we prove a similar result to Theorem \ref{thm:BvHL} which characterises graphs having an eigenvalue multiplicities $2, n-2$.

\begin{theorem}\label{thm:rk2}
Let $G$ be a graph on $n$ vertices. Then there exists a matrix $A$ with $G(A)=G$ and with a  characteristic polynomial $p(x)=x^{n-2}(x-a)^2$ if and only if $G^c$ has the form
 \begin{equation}\label{n-2form}
 (K_{p_1,q_1}\cup K_{p_2,q_2} \cup \ldots \cup K_{p_k,q_k})\vee K_r
 \end{equation}
 for appropriate nonnegative integers $k, p_1,q_1,\ldots,p_k,q_k,r$ with $p_i+q_i>0,$ $i=1,2,\ldots,k$.
\end{theorem}

\begin{proof}
Necessity follows from Theorem \ref{thm:BvHL}. To prove sufficiency 
we modify the proof of Theorem 3 in \cite{MR2111528}. 

Let $G$ be a graph of the form (\ref{n-2form}). Let $A$ be a positive semidefinite matrix in with rank $2$ and the characteristic polynomial $p(x)=x^{n-2}(x-a)^2$. 
We will show there exits a $2 \times n$ matrix $U=[u_1, u_ 2, \ldots ,u_{n}],$ $u_i \in \R^2,$ such that $A=U^TU$ and  
$$UU^T=\npmatrix{a & 0 \\ 0 & a},$$  
since $U^TU$ and $UU^T$ have the same nonzero spectrum. 

We define $U$ in the following way. 
For each $i \in V(K_r)$, let $u_i=0.$
 Let $S_i,T_i$ be the color classes of $K_{p_i,q_i}$, $i=1,2,\ldots,k$. For $v \in S_i,$ define $$u_v=\frac{1}{\sqrt{p_i}}\npmatrix{i \\ 1}$$ and for $v \in T_i$ define $$u_v=\frac{1}{\sqrt{q_i}}\npmatrix{1 \\ -i}.$$
It is not difficult to check that this choice of $U$ gives us $G(U^TU)=G$. Since 
\begin{align*} 
 UU^T&=\npmatrix{\sum_{i=1}^k(p_i\frac{i^2}{p_i}+q_i\frac{1}{q_i}) & 
  \sum_{i=1}^k(p_i\frac{i}{p_i}-q_i\frac{i^2}{q_i}) \\
   \sum_{i=1}^k(p_i\frac{i}{p_i}-q_i\frac{i^2}{q_i}) & \sum_{i=1}^k(p_i\frac{1}{p_i}+q_i\frac{i^2}{q_i}) }\\
 &=  \npmatrix{\sum_{i=1}^k(i^2+1) & 0 \\ 0 & \sum_{i=1}^k(i^2+1) }
\end{align*}    
 the characteristic polynomial of $U^TU$ is equal to $(x-a)^2x^{n-2}$ for  $a=\sum_{i=1}^k(i^2+1)$.
 \end{proof}

\begin{example}
By Theorem \ref{thm:rk2} the following  $11$  graphs on 6 vertices allow characteristic polynomial a square:
 
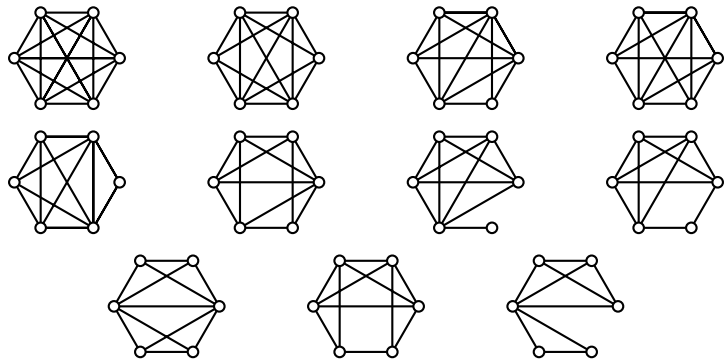
\begin{center} 
    \begin{tikzpicture}[style=thick, scale=0.7]
		\draw \foreach \x in {0,60,120,180,240,300} {
				(\x:1) node{} -- (\x+60:1) node{}
				(\x:1) node{} -- (\x+120:1) node{}				
				(\x:1) node{} -- (\x+180:1) node{}				
		};		
		\draw[fill=white] \foreach \x in {0,60,120,180,240,300} {
		                    (\x:1) circle (1mm)
		};		
   \end{tikzpicture}
\qquad
    \begin{tikzpicture}[style=thick, scale=0.7]
		\draw \foreach \x in {0,60,120,180,240,300} {
				(\x:1) node{} -- (\x+60:1) node{}
				(\x:1) node{} -- (\x+120:1) node{}				
		};		
		\draw (60:1) node{} -- (240:1) node{}	;
		\draw (120:1) node{} -- (300:1) node{}	;		
		\draw[fill=white] \foreach \x in {0,60,120,180,240,300} {
		                    (\x:1) circle (1mm)
		};		
    \end{tikzpicture}
\qquad
    \begin{tikzpicture}[style=thick, scale=0.7]
		\draw \foreach \x in {0,60,120,180,240,300} {
				(\x:1) node{} -- (\x+60:1) node{}				
				(\x:1) node{} -- (60:1) node{}			};		
		\draw \foreach \x in {0,120,240} {
				(\x:1) node{} -- (\x+120:1) node{}				
		};		
		\draw (0:1) node{} -- (180:1) node{};
		\draw[fill=white] \foreach \x in {0,60,120,180,240,300} {
		                    (\x:1) circle (1mm)
		};		
    \end{tikzpicture}
\qquad
    \begin{tikzpicture}[style=thick, scale=0.7]
		\draw \foreach \x in {0,60,120,180,240,300} {
				(\x:1) node{} -- (\x+60:1) node{}				
				(\x:1) node{} -- (60:1) node{}			};		
		\draw \foreach \x in {120,240} {
				(\x:1) node{} -- (\x+120:1) node{}				
		};		
		\draw (0:1) node{} -- (180:1) node{};
		\draw (120:1) node{} -- (300:1) node{};
		\draw[fill=white] \foreach \x in {0,60,120,180,240,300} {
		                    (\x:1) circle (1mm)
		};		
   \end{tikzpicture}
\end{center}
\begin{center}
    \begin{tikzpicture}[style=thick, scale=0.7]
		\draw \foreach \x in {0,60,120,180,240,300} {
				(\x:1) node{} -- (\x+60:1) node{}				
				(\x:1) node{} -- (60:1) node{}	
				(\x:1) node{} -- (300:1) node{}			};		
		\draw (120:1) node{} -- (240:1) node{};
		\draw[fill=white] \foreach \x in {0,60,120,180,240,300} {
		                    (\x:1) circle (1mm)
		};		
    \end{tikzpicture}
\qquad
    \begin{tikzpicture}[style=thick, scale=0.7]
		\draw \foreach \x in {0,60,120,180,240,300} {
				(\x:1) node{} -- (\x+60:1) node{}	   };
		\draw \foreach \x in {0,120,240} {
				(\x:1) node{} -- (\x+120:1) node{}	};			
		\draw (300:1) node{} -- (60:1) node{} -- (180:1) node{};
		\draw (0:1) node{} -- (180:1) node{};
		\draw[fill=white] \foreach \x in {0,60,120,180,240,300} {
		                    (\x:1) circle (1mm)
		};		
    \end{tikzpicture}
\qquad
    \begin{tikzpicture}[style=thick, scale=0.7]
		\draw \foreach \x in {0,60,120,180,240} {
				(\x:1) node{} -- (\x+60:1) node{}	   };
		\draw \foreach \x in {0,60,120} {
				(\x:1) node{} -- (\x+120:1) node{}	   };
		\draw (60:1) node{} -- (240:1) node{};
		\draw (240:1) node{} -- (0:1) node{} -- (180:1) node{};
		\draw[fill=white] \foreach \x in {0,60,120,180,240,300} {
		                    (\x:1) circle (1mm)
		};		
     \end{tikzpicture}
\qquad
    \begin{tikzpicture}[style=thick, scale=0.7]
		\draw \foreach \x in {0,60,120,180,240,300} {
				(\x:1) node{} -- (\x+60:1) node{}	   };
		\draw \foreach \x in {0,60,120} {
				(\x:1) node{} -- (\x+120:1) node{}	   };
		\draw (60:1) node{} -- (240:1) node{};
		\draw (0:1) node{} -- (180:1) node{};
		\draw[fill=white] \foreach \x in {0,60,120,180,240,300} {
		                    (\x:1) circle (1mm)
		};		
   \end{tikzpicture}
\end{center}
\begin{center}
    \begin{tikzpicture}[style=thick, scale=0.7]
		\draw \foreach \x in {0,60,120,180,240,300} {
				(\x:1) node{} -- (\x+60:1) node{}	   };
		\draw \foreach \x in {0,60,180,240} {
				(\x:1) node{} -- (\x+120:1) node{}	   };
		\draw (0:1) node{} -- (180:1) node{};
		\draw[fill=white] \foreach \x in {0,60,120,180,240,300} {
		                    (\x:1) circle (1mm)
		};		
   \end{tikzpicture}
\qquad
    \begin{tikzpicture}[style=thick, scale=0.7]
		\draw \foreach \x in {0,60,120,180,240,300} {
				(\x:1) node{} -- (\x+60:1) node{}	   };
		\draw \foreach \x in {0,60,120,300} {
				(\x:1) node{} -- (\x+120:1) node{}	   };
		\draw (0:1) node{} -- (180:1) node{};
		\draw[fill=white] \foreach \x in {0,60,120,180,240,300} {
		                    (\x:1) circle (1mm)
		};		
    \end{tikzpicture}
\qquad
    \begin{tikzpicture}[style=thick, scale=0.7]
		\draw \foreach \x in {0,60,120,180,240} {
				(\x:1) node{} -- (\x+60:1) node{}	   };
		\draw \foreach \x in {0,60,180} {
				(\x:1) node{} -- (\x+120:1) node{}	   };
		\draw (0:1) node{} -- (180:1) node{};
		\draw[fill=white] \foreach \x in {0,60,120,180,240,300} {
		                    (\x:1) circle (1mm)
		};		
   \end{tikzpicture}
\end{center}
\end{example}

\section{Graphs with $4$ vertices}

We conclude by  listing all nonisomorhic simple connected graphs on $4$ vertices.

\begin{example}
There are $6$ nonisomorhic simple connected graphs on $4$ vertices. Graphs
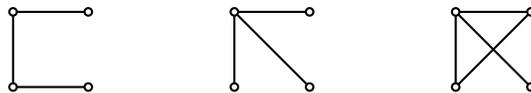
\begin{center}
 \begin{tikzpicture}[style=thick,scale=0.5]
	\draw (1,1) -- (-1,1) -- (-1,-1) -- (1,-1);
	\draw[fill=white] (-1,1) circle (1mm) (-1,-1) circle (1mm) (1,-1) circle (1mm) (1,1) circle (1mm);
\end{tikzpicture}  
\qquad
\qquad
 \begin{tikzpicture}[style=thick,scale=0.5]
	\draw (1,1) -- (-1,1) -- (-1,-1);
	\draw (-1,1) -- (1,-1);
	\draw[fill=white] (-1,1) circle (1mm) (-1,-1) circle (1mm) (1,-1) circle (1mm) (1,1) circle (1mm);
\end{tikzpicture}  
\qquad
\qquad
 \begin{tikzpicture}[style=thick,scale=0.5]
	\draw (1,1) -- (-1,1) -- (-1,-1) -- (1,1);
	\draw (-1,1) -- (1,-1);
	\draw[fill=white] (-1,1) circle (1mm) (-1,-1) circle (1mm) (1,-1) circle (1mm) (1,1) circle (1mm);
\end{tikzpicture}  
\end{center}
do not allow the characteristic polynomial a square by Proposition \ref{tree} and by Example \ref{cycle+1}. However, graphs

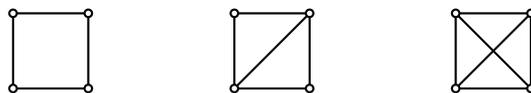
\begin{center}
 \begin{tikzpicture}[style=thick,scale=0.5]
	\draw (1,1) -- (-1,1) -- (-1,-1) -- (1,-1) -- (1,1);
	\draw[fill=white] (-1,1) circle (1mm) (-1,-1) circle (1mm) (1,-1) circle (1mm) (1,1) circle (1mm);
\end{tikzpicture}  
\qquad
\qquad
 \begin{tikzpicture}[style=thick,scale=0.5]
	\draw (1,1) -- (-1,1) -- (-1,-1) -- (1,-1) -- (1,1) -- (-1,-1);
	\draw[fill=white] (-1,1) circle (1mm) (-1,-1) circle (1mm) (1,-1) circle (1mm) (1,1) circle (1mm);
\end{tikzpicture}  
\qquad
\qquad
 \begin{tikzpicture}[style=thick,scale=0.5]
	\draw (1,1) -- (-1,1) -- (-1,-1) -- (1,-1) -- (1,1) -- (-1,-1);
	\draw (-1,1) -- (1,-1);
	\draw[fill=white] (-1,1) circle (1mm) (-1,-1) circle (1mm) (1,-1) circle (1mm) (1,1) circle (1mm);
\end{tikzpicture}  
\end{center}
allow the characteristic polynomial a square by  Corollary \ref{cycle}, Theorem \ref{thm:rk2} and  Corollary \ref{thm:Kn}.
\end{example}


\begin{thebibliography}{10}

\bibitem{2013arXiv1304.1205A}
B.~{Ahmadi}, F.~{Alinaghipour}, M.~S. {Cavers}, S.~{Fallat}, K.~{Meagher}, and
  S.~{Nasserasr}.
\newblock {Minimum number of distinct eigenvalues of graphs}.
\newblock {\em ArXiv e-prints}, April 2013.

\bibitem{MR2547901}
Wayne Barrett, H.~Tracy Hall, and Raphael Loewy.
\newblock The inverse inertia problem for graphs: cut vertices, trees, and a
  counterexample.
\newblock {\em Linear Algebra Appl.}, 431(8):1147--1191, 2009.

\bibitem{MR2781688}
Wayne Barrett, H.~Tracy Hall, and Hein van~der Holst.
\newblock The inertia set of the join of graphs.
\newblock {\em Linear Algebra Appl.}, 434(10):2197--2203, 2011.

\bibitem{MR2596445}
Wayne Barrett, Camille Jepsen, Robert Lang, Emily McHenry, Curtis Nelson, and
  Kayla Owens.
\newblock Inertia sets for graphs on six or fewer vertices.
\newblock {\em Electron. J. Linear Algebra}, 20:53--78, 2010.

\bibitem{MR2111528}
Wayne Barrett, Hein van~der Holst, and Raphael Loewy.
\newblock Graphs whose minimal rank is two.
\newblock {\em Electron. J. Linear Algebra}, 11:258--280, 2004.

\bibitem{MR2996934}
Aleksandra Eri{{\'c}} and C.~M. da~Fonseca.
\newblock Unordered multiplicity lists of wide double paths.
\newblock {\em Ars Math. Contemp.}, 6(2):279--288, 2013.

\bibitem{2011arXiv1102.5142F}
S.~{Fallat} and L.~{Hogben}.
\newblock {Variants on the minimum rank problem: A survey II}.
\newblock {\em ArXiv e-prints}, February 2011.

\bibitem{MR2350678}
Shaun~M. Fallat and Leslie Hogben.
\newblock The minimum rank of symmetric matrices described by a graph: a
  survey.
\newblock {\em Linear Algebra Appl.}, 426(2-3):558--582, 2007.

\bibitem{MR1902112}
Charles~R. Johnson and Ant{{\'o}}nio Leal~Duarte.
\newblock On the possible multiplicities of the eigenvalues of a {H}ermitian
  matrix whose graph is a tree.
\newblock {\em Linear Algebra Appl.}, 348:7--21, 2002.

\bibitem{MR3005270}
Charles~R. Johnson, Jonathan Nuckols, and Calum Spicer.
\newblock The implicit construction of multiplicity lists for classes of trees
  and verification of some conjectures.
\newblock {\em Linear Algebra Appl.}, 438(5):1990--2003, 2013.

\bibitem{MR3034496}
In-Jae Kim and Bryan~L. Shader.
\newblock Unordered multiplicity lists of a class of binary trees.
\newblock {\em Linear Algebra Appl.}, 438(10):3781--3788, 2013.

\bibitem{MR2098598}
Helena {\v{S}}migoc.
\newblock The inverse eigenvalue problem for nonnegative matrices.
\newblock {\em Linear Algebra Appl.}, 393:365--374, 2004.

\end{thebibliography}
\end{document}